\documentclass[11pt]{article}

\usepackage{tgtermes}
\usepackage[T1]{fontenc}
\usepackage{amsmath,amsthm}
\usepackage{graphicx}
\usepackage{xcolor}

\usepackage{arxivs}
\usepackage[utf8]{inputenc}
\usepackage{amsfonts,amssymb}

\usepackage{cite}
\usepackage{hyperref}
\usepackage{textcomp}
\usepackage{tabularx}
\usepackage{caption}
\usepackage{diagbox}

\usepackage{tabularx}
\usepackage{diagbox}
\usepackage{amscd}
\usepackage{float}
\usepackage{placeins}
\usepackage{mathrsfs}%
\usepackage[title]{appendix}%
\usepackage{textcomp}%
\usepackage{manyfoot}%
\usepackage{booktabs}%
\usepackage{algorithm}%
\usepackage{algorithmicx}%
\usepackage{algpseudocode}%
\usepackage{listings}%
\usepackage{float}
 \usepackage{setspace}
\newtheorem{theorem}{Theorem}
\newtheorem{proposition}[theorem]{Proposition}%
\newtheorem{example}{Example}%
\newtheorem{definition}{Definition}%

\usepackage{multirow}     
\usepackage{float}        
\usepackage{titlesec}     
\usepackage{bbm}          
\usepackage{subfig}       
\usepackage{array}        
\usepackage[numbers]{natbib}
 \usepackage{setspace}
\newcommand{\Zc}{\mathcal{Z}}
\newcommand{\Fc}{\mathcal{F}}
\newcommand{\Lc}{\mathcal{L}}
\newcommand{\Ic}{\mathcal{I}}
\newcommand{\Jc}{\mathcal{J}}
\newcommand{\Ac}{\mathcal{A}}
\newcommand{\Bc}{\mathcal{B}}
\newcommand{\Sc}{\mathcal{S}}
\newcommand{\Xc}{\mathcal{X}}
\newcommand{\Yc}{\mathcal{Y}}

\newcommand{\Nc}{\mathcal{N}}
\newcommand{\Xf}{\mathfrak{X}}

\newcommand{\Rb}{\mathbb{R}}
\newcommand{\Eb}{\mathbb{E}}
\newcommand{\Ab}{\mathbb{A}}
\newcommand{\onen}{\mathbf{1}}
\newcommand{\one}{\mathbbm{1}}

\newcommand{\sys}{\text{sys}}
\newcommand{\lin}{\text{lin}}

\newcommand{\Dc}{\mathcal{D}}

\renewcommand{\kappa}{\varkappa}

\DeclareMathOperator{\avar}{AVaR} 
\DeclareMathOperator{\icx}{icx}

\newenvironment{tightlist}[1]{%
    \list{{\rm(\roman{enumi})}}{\settowidth\labelwidth{{\rm(#1)}}
    \leftmargin\labelwidth \advance\leftmargin\labelsep
    \parsep 0pt plus 1pt minus 1pt \topsep 0pt \itemsep 0pt
    \usecounter{enumi}}}{\endlist}

\newenvironment{tightitemize}{%
    \list{{\textup{$\bullet$}}}{\settowidth\labelwidth{{\textup{\qquad}}}
    \leftmargin\labelwidth \advance\leftmargin\labelsep
    \parsep 0pt plus 1pt minus 1pt \topsep 3pt \itemsep 3pt
     }}{\endlist}

\title{Fair Risk Optimization of Distributed Systems}
\author{Aray Almen \\
        Department of Mathematical sciences \\
        Stevens Institute of Technology \\
        Hoboken, NJ 07030, USA \\
        \texttt{aalmen@stevens.edu} \\
        \And 
        Darinka Dentcheva \\
        Department of Mathematical sciences \\
        Stevens Institute of Technology \\
        Hoboken, NJ 07030, USA \\
        \texttt{darinka.dentcheva@stevens.edu}
        }

\date{}

\begin{document}
\maketitle

\begin{abstract}
    The paper provides a framework for the assessment and optimization of the total risk of complex distributed systems. The framework takes into account the risk of each agent, which may arise from heterogeneous sources, as well as the risk associated with the efficient operation of the system as a whole.  
    The challenges posed by this task are associated with the lack of additivity of risk, the need to evaluate the risk of every agent (unit) using confidential or proprietary information, and the requirement of fair risk allocation to agents (units). 
    We analyze systemic risk measures that are based on a sound axiomatic foundation  while at the same time facilitate risk-averse sequential decision-making by distributed numerical methods, which allow the agents to operate autonomously with minimal exchange of information. 
    
    We formulate a two-stage decision problem for a distributed system using systemic measures of risk and devise a decomposition method for solving the problem. The method is applied to a disaster management problem. We have paid attention to maintain fair risk allocation to all areas in the course of the relief operation. Our numerical results show the efficiency of the proposed methodology. 
\end{abstract}



\section{Introduction}\label{intro}

    The main focus of our study is the question of risk evaluation and control for complex distributed systems. This type of system arises in many areas that involve multiple agents or units/subsystems that operate autonomously but also depend on the proper operation of the entire system. In this case, the assessment of total risk should consider the risk associated with each agent as well as the risk related to the integrity and efficient operation of the system as a whole. 
    While a large body of published research focuses on the properties of risk measures and their application in finance, less work addresses risk control in distributed energy systems, business systems, logistics problems, or robotic networks, where heterogeneous sources of risk may exist and complex relationships govern operation. 

    Further challenges arise when the risk evaluation for the subsystems is based on confidential or proprietary information. In some applications, it also becomes essential to distribute risk fairly among the agents or units. At some level, proper risk evaluation and budgeting for the entire system, along with fair risk allocation to its units, are essential when regulatory policies are designed.   
    The main objective of this paper is to address the evaluation and risk optimization of distributed complex systems that present these challenges. While building on the developments thus far, our goal is to advance a framework that is both theoretically sound and amenable to efficient numerical computations. 

    The evaluation of risk for a distributed complex system is non-trivial because the risk is not additive. This risk is termed ``systemic'' and the measures for it are called systemic risk measures. We also note that the risk of the system may include components characterizing the system as a whole and not purely aggregating the risk of the individual units or agents. For example, this may be the risk associated with completing a common task in cooperative systems.  
    A significant body of literature focuses on systemic risk measures designed for specific financial applications. In \cite{Brunnermeier-Cheridito}, the authors introduce a risk measure called SystRisk, which quantifies the risk as the additional amount of money needed to make the total externality of a financial network exceed some predefined threshold level. In \cite{covar}, the concept of CoVaR is introduced, which quantifies the contribution of financial institutions to systemic risk using conditional quantiles.

    Suppose the system of interest consists of $m$ units (agents).  
    Two major approaches to evaluating the total risk can be distinguished. One approach is to use an \textit{aggregation function}, $\Lambda:\Rb^m\to \Rb$, to aggregate the loss functions of the agents. Then the result is evaluated by a univariate risk measure. This type of analysis is presented in \cite{Chen-Iyengar} for finite probability spaces. In \cite{Kromer}, the authors propose aggregation functions, which are particularly suitable for financial systems. They also analyze convex risk measures defined on general probability spaces. In both studies, the aggregation function $\Lambda$ must satisfy properties similar to the axioms postulated for risk measures. One can also analyze the maximal risk over a class of aggregation functions rather than using one specific function. We refer to \cite{Rusch-risk-analysis-book} for an overview of the risk measures constructed this way.

    Another approach to risk evaluation of complex systems consists of evaluating agents' individual risks first and then aggregating the obtained values. This method is used, for example, in \cite{Biagini-v1} and in \cite{Feinstein-Rudloff}. In \cite{kiesel2010optimal}, convex risk functionals are defined for portfolios of risk vectors following this aggregation principle. A further extension in \cite{Biagini-v1} uses a set of admissible risk allocations for the units and suggests a risk measure that allows for scenario-dependent target allocations, where determining the total acceptable loss is done ahead of time. In \cite{Biagini-v2}, the authors analyze the existence and uniqueness of the optimal allocation resulting from those systemic risk measures and bring up the issue of fairness in risk allocation. They argue that the optimizer in the dual formulation provides a risk allocation that is fair from the point of view of the individual financial institutions. In this paper, we intend to present an alternative perspective on fairness in risk allocation. In \cite{Feinstein-Rudloff}, a set-valued counterpart of this approach is proposed by defining the systemic risk measure as the set of all vectors that make the outcome acceptable. Once the set of all acceptable allocations is constructed, one can derive a scalar-valued \textit{efficient allocation rule} by minimizing the weighted sum of components of the vectors in the set. 
    Set-valued risk measures were proposed in \cite{Jouini-Meddeb}, see also \cite{Ararat-Rudloff,Hamel-Heyde-dual} for duality theory, including the dual representation for specific set-valued risk measures. 
    An important feature of the aggregation function $\Lambda$ is to capture the interdependence between the system's components. In \cite{Pflug-copula-risk}, the authors propose to use the copula function for aggregation. The thesis made there is that independent operation does not carry systemic risk, as the risk of the individual components or agents can be optimized independently.

    Some work includes methods that use some multivariate counterpart of the univariate risk measures, most notably the notion of Multivariate Value-at-Risk, which is a counterpart to Value-at-Risk for random vectors
    and is identified with the set of \textit{$p$-efficient points}. The latter plays a role in stochastic programming with chance constraints (see, e.g. \cite{Dentcheva_Lai_Rusz,SP-textbook, prekopa2013stochastic}).
    Multivariate Value-at-Risk is used to define Average Value-at-Risk for multivariate distributions in various ways; see \cite{Lee-Prekopa,Noyan-Rudolf-cvar,Prekopa,Merakli,doldi2024shortfall}. Further extension pertains to the inclusion of quasiconvex risk measures, in particular, in the context of risk sharing and allocation \cite{mastrogiacomo2015pareto}.

    Axiomatic approaches to systemic risk are proposed in \cite{almen2024risk,Burgert-Rusch,Ruschendorf-dual,Chen-Iyengar,Ekeland-Law}. In our previous work \cite{almen2024risk}, we put forward a set of axioms for functionals defined on the space of random vectors. The point of view is that the risk factors of various sources and/or the loss functions of each unit are comprised in a random vector, which is in-line with the premise of \cite{Burgert-Rusch,Ruschendorf-dual,Chen-Iyengar,Ekeland-Law,Feinstein-Rudloff} among others. The set of axioms put forward in our earlier paper \cite{almen2024risk} are most closely related to the axioms proposed in \cite{Burgert-Rusch} and in \cite{Ekeland-Law}. A version of the translation property for random vectors was first introduced in \cite{Burgert-Rusch} formulated for convex risk measures that are defined on the space of bounded random vectors. This property requires translation to hold for all deterministic vectors, which is substantially different from the translation in \cite{almen2024risk,Ekeland-Law}. In \cite{Ekeland-Law}, the authors consider law-invariant risk measures for bounded random vectors to establish a Kusuoka representation for random vectors. In their analysis, the authors also require two normalization properties along with the axioms presented in \cite{almen2024risk}. The axioms in the latter paper appear to be the minimal set of assumptions needed to obtain a dual representation of the systemic risk measures. We shall provide relevant details in the following section.

    While most of the work on systemic risk is focused on analysis of a financial system or a set of portfolios for different branches of a firm, less work addresses other systems. Recently, several areas of research such as machine learning methods, robotics, and reinforcement learning have included measures of risk in their models. This interest is motivated by the robustness properties of the risk measures in the context of data-driven optimization. Optimization problems arising in statistical learning and robotics frequently require distributed optimization methods when addressing multi-agent systems. We mention a few references as it is impossible to be exhaustive; see \cite {AP-Selection-DD,vitt2019risk,zhu2025mean,Li2024distributed,wu2021impact,ma2017risk,lyu2023risk}.

    In this work we extend the analysis and numerical treatment of systemic risk for multi-agent systems. We propose a new risk-averse two-stage structured optimization problem that reflects the operation of a distributed system. We devise a numerical decomposition method. The system is disaggregated to allow for local control of each unit, minimizing its own risk, while cooperating in the process of minimizing the risk of the system as a whole. The method is implemented to solve a disaster management problem, which deals with the optimal allocation of resources in selected areas to minimize the damage caused by the disaster. The problem is similar to the two-stage problem in \cite{Noyan-disaster-problem}, however, it allows for local management decisions aiming at minimizing the risk locally, while the resource allocation comes from a global authority. This type of situation arises when federal agencies support local efforts to mitigate the effects of hurricanes, floods, fires, and other disasters.

    Our paper is organized as follows. 
    Section \ref{s:analysis} presents the analysis of the systemic risk measures, which evaluate the risk of individual components first and then aggregate the risk to determine the total risk of the entire system. We also present new results regarding the subdifferentiability of the systemic risk measures in question and their dual representation. 
    In Section \ref{s:problem}, we formulate a risk-averse two-stage stochastic programming problem and propose a distributed method for solving it. Section \ref{s:numerical_results} contains the numerical experiments with the disaster management problem.

\section{Risk measures for high-dimensional risks}
\label{s:analysis}

    First, recall  that the risk of a loss function can be evaluated using a univariate coherent measure of risk according to the widely accepted axiomatic framework that is proposed in \cite{Artzner} and further analyzed in \cite{Delbaen,Follmer,RS:2005,RuSh:2006a,PflRom:07,DDARriskbook}, and many other works. We refer to \cite{DDARriskbook} for an extensive treatment of risk measures and stochastic optimization with such measures. 

    Let $\Lc_p(\Omega, \Fc, P)$ denote the space of real-valued random variables defined on the probability space $(\Omega, \Fc, P)$ with finite $p$-th moments, $p\in[1,\infty)$, that are indistinguishable on events with zero probability. We shall assume that the random variables represent random costs, and we prefer small outcomes over large ones. The space  $\Lc_p(\Omega, \Fc, P)$ is equipped with the norm topology. It is paired with $\Lc_q(\Omega, \Fc, P)$, where $\frac{1}{p} + \frac{1}{q} = 1$. For any $Z \in \Lc_p(\Omega, \Fc, P) $ and $\xi \in \Lc_q(\Omega, \Fc, P)$, we use the bilinear form
    \[  
        \langle \xi, Z \rangle = \int_\Omega \xi(\omega) Z(\omega) dP(\omega). 
    \]
     The space  $\Lc_q(\Omega, \Fc, P)$ is equipped with a topology compatible with the pairing.
     Recall that a lower semi-continuous proper functional $\varrho: \Lc_p(\Omega, \Fc, P) \to \Rb\cup\{+\infty\}$ is a \textit{coherent risk measure} if it is convex, positively homogeneous, monotonic with respect to the a.s. comparison of random variables, and satisfies the following translation property:
    \[
        \varrho[Z + a] = \varrho[Z] + a \text{ for all } Z \in \Lc_p(\Omega, \Fc, P),\; a \in \Rb. 
    \]
    If $\varrho[\cdot]$ is monotonic, convex, and satisfies the translation property, it is called a \emph{convex} risk measure. 
    Every proper lower semicontinuous coherent risk measure $\varrho$ has a dual representation of the form
    \begin{equation}
        \varrho[Z] = \sup_{\xi \in \Ac_\varrho} \langle \xi, Z \rangle, \quad Z \in \Lc_p(\Omega, \Fc, P),
        \label{e:dual-rep-rho}
    \end{equation}
    where $\Ac_\varrho \subset \{ \xi \in \Lc_q(\Omega, \Fc, P) ~ | ~ \xi \geq 0 \text{ a.s.}, ~ \int_\Omega \xi(\omega)dP(\omega) = 1 \}$ is the convex-analysis subdifferential $\partial \varrho[0]$. 

    As our goal is to address high-dimensional sources of risks arising in complex distributed systems where agents or units operate semi-autonomously.  We focus on the space  $\Xf = \Lc_p(\Omega, \Fc, P; \Rb^m)$ of random vectors with finite $p$-th moments ($p\in[1,\infty]$) whose realizations are in $\Rb^m$. 

    The $m$-dimensional vector, whose components are all equal to one is denoted by $\onen$, and the random vector with realizations equal to $\onen$ is denoted by $\one$. The following definition is introduced and analyzed in \cite{almen2024risk}, see also \cite{DDARriskbook}: 
    \begin{definition}
    \label{d:riskonvectors}
        A lower semi-continuous functional $\varrho: \Xf \to \Rb\cup\{+\infty\}$ is a systemic coherent risk measure with preference to small outcomes iff it satisfies the following properties:
        \begin{description}
            \item[{\rm A1.}] {\rm Convexity:} For all $X, Y \in \Xf$ and for all $ \alpha \in (0,1)$, the following inequality holds\\
             \hspace*{4ex} $\varrho[\alpha X + (1-\alpha)Y] \leq \alpha \varrho[X] + (1 - \alpha)\varrho[Y] $.
            \item[{\rm A2.}] {\rm  Monotonicity:} For all $X, Y \in \Xf$, if $X_i \geq Y_i$ $P$-a.s. for all $i = 1, \dots, m$, then \hspace*{4ex} $\varrho[X] \geq \varrho[Y]$.
            \item[{\rm A3.}] {\rm  Translation property:} For all $X \in \Xf$ and $a \in \Rb$, $\varrho[X + a\one] = \varrho[X] + a\varrho[\one].$
            \item[{\rm A4.}] {\rm  Positive homogeneity:} For all $X \in \Xf$ and $t > 0$, we have $\varrho[tX] = t\varrho[X]$.
        \end{description}
    \end{definition}
    For $p<\infty$, it is shown in \cite{almen2024risk} that if the systemic risk measure $\varrho$ is proper, lower semicontinuous, and satisfies those axioms, then it has the following dual representation:
    \begin{equation}
        \varrho[X] = \sup_{\zeta \in \Ac_\varrho} \big\{ \langle \zeta, X \rangle_\Xf \big\}.
        \label{general dual}
    \end{equation}
    Here $\langle \zeta, X\rangle_\Xf$ denotes the dual pairing between $\Xf$  and $\Xf^*$ and the set $\Ac_\varrho\subset \Lc_q(\Omega, \Fc, P;\Rb^m)$ with $\frac{1}{p} + \frac{1}{q} =1$ is defined as:
    \begin{equation}
    \label{e:dualset-gen}
        \Ac_\varrho = \Big\{ \zeta \in \Lc_q(\Omega, \Fc, P;\Rb^m): \int_\Omega \zeta(\omega) dP(\omega) = \mu_\zeta, ~ \zeta \geq 0 \text{ a.s.},\;  \langle \one, \mu_\zeta \rangle =\varrho(\one) \Big\}.
    \end{equation}
    The proof of representation \ref{general dual} relies on the dual pairing of $\Xf$  and $\Xf^*$ so that
    all continuous linear functionals on $\Xf$ and $\Xf^*$ are given by the mappings $Z\mapsto \langle \zeta,Z \rangle $ (for a fixed $\zeta\in\Xf^*$) and $\zeta\mapsto \langle \zeta,Z \rangle $ (for a fixed $Z\in\Xf$) respectively. For $p\in (0,1),$ all continuous linear functionals on $\Zc^*$ (equipped with its norm topology) have this form. For $p=1$, we equip $Z^*$ with the weak$^*$ topology. 
    For $p = \infty$,  we can pair $\Zc = \Lc_\infty(\Omega, \Fc,P,\Rb^2)$ with $\Lc_1(\Omega, \Fc,P,\Rb^2)$ and equip the latter space with its norm topology and the former with its weak$^*$ topology. In that case, we also need the additional assumption that $\varrho$ is lower-semicontinuous with respect to its weak$^*$ topology, which is a very strong assumption. 
    \begin{definition}
    A systemic measure of risk $\varrho(\cdot)$ is normalized if $\varrho[\one] = 1$.
    \end{definition} 
    If the systemic measure $\varrho(\cdot)$ satisfies (A1)--(A4) and it is normalized, then $r =1$ for all $\zeta \in \Ac_\varrho$.
    This entails that for all $\zeta \in \Ac_\varrho$, $\zeta P$ can be interpreted as a probability measure on the space $\Omega \times \{ 1, 2 ,\dots, m \}$.  In the special case when $m=1$, we obtain the widely used dual representation of coherent measures of risk for scalar-valued random variables 
    \begin{equation}
        \varrho[X] = \sup_{\frac{dQ}{dP} \in \Ac_\varrho} \Eb_Q [ X ],
        \label{scalar dual}
    \end{equation}
    where $\frac{dQ}{dP}$ is the Radon-Nikodym derivative of the measure $Q$ with respect to the reference measure $P$.
    Hence, the dual representation shows that the risk measure takes into account the ``worse'' expectation of random vectors taken with respect to measures that augment the original probability measure $P$ and are absolutely continuous with respect to it.   
    Let $S^m_+$ denote the simplex of $m$-dimensional scalarization vectors:
    \[
        S^m_+ = \{ c\in \Rb^m : \sum_{i=1}^m c_i = 1, \; c_i \geq 0, \; i=1,\dots,m\}.
    \]
    Following the two principles of aggregation of risk, we have proposed two ways of constructing systemic measures of risk. 
    The first construction fixes a closed set $S\subset S^m_+$ and a lower semi-continuous univariate risk measure $\varrho:\Lc_p(\Omega,\Fc, P)\to\Rb\cup\{ +\infty\}$. The systemic risk measure $\varrho_S(\cdot)$ aggregates outcomes first and then evaluates the risk according to the following method. 
    \begin{equation}
        \label{rho_max_S}
        \varrho_S[X]=\varrho[M_S(X)],\quad\text{where } [M_S(X)] (\omega)= \max_{c\in S} c^\top X(\omega),\;\; \omega\in\Omega.
    \end{equation}
    It is straightforward to see that the systemic risk measure $\varrho_S(\cdot)$ is well-defined on $\Xf.$ We have shown that it is coherent (convex) according to Definition~\ref{d:riskonvectors} provided $\varrho(\cdot)$ is a coherent (convex) univariate measure of risk. This type of aggregation method is unsuitable when privacy issues are associated with the operation of the units or when proprietary information is involved. In that case, the method of evaluating risk for the agents or units first and then using aggregation is better suited to systemic risk management.
    Furthermore, in many situations, the issue of fairness in risk allocation among the agents or units becomes essential. This is why we proposed a new way of risk aggregation, which functions as follows. 
    Let $\Omega_m = \{1,\dots,m\}$ and $c\in S^m_+$. Then we view $c$ as a probability mass function on the space $\Omega_m$ and consider the probability space $(\Omega_m, \Fc_m, c)$, where $\Fc_m$ contains all subsets of $\Omega_m$.
    \begin{definition}
        Given a random vector $X\in\Xf$, a collection of $m$ proper risk measures $\varrho_i: \Lc_p(\Omega, \Fc, P)\to\Rb\cup\{+\infty\}$, $i=1,\dots,m,$ and a risk measure $\varrho_0 : \Lc_p(\Omega_m, \Fc_m, c) \to \Rb$, we define the systemic risk measure $\varrho_\sys : \Xf \to \Rb$ as follows. 
        \begin{equation}
            \varrho_\sys(X) = \varrho_0[R(X)],
            \label{def:syst-risk}
        \end{equation}
        where $R(X)\in\Lc_p(\Omega_m, \Fc_c, c)$ is a random variable with $m$ realizations given by 
        \[
        [R(X)](i) = \varrho_i(X_i)\quad i=1,\dots,m.
        \]
    \end{definition}
    Notice that, we could identify $\Lc_p(\Omega_m, \Fc_m, c)$ with $\Rb^m$ equipped with any of its norms. We stick to the notation $\Lc_p(\Omega_m, \Fc_m, c)$ for clarity and consistency; it also plays a role in our further analysis. 
    In a financial context, we could interpret this risk measure in the following way. Consider a firm with several branches of operation such as different portfolios on various markets, insurance contracts, etc. These portfolios have their risk evaluation while a risk aggregation for the total risk of the firm's operation is provided by  $\varrho_{\rm sys}(\cdot)$.  This way of aggregation of the portfolio risks provides a fair risk allocation to each portfolio/branch while providing robust evaluation of the total risk for the firm as we shall see later.  

    We adopt the usual definition of law-invariance.  If two random vectors $X$ and $Y$ have the same distribution, then a law-invariant risk measure $\varrho(\cdot)$ provides the same risk evaluation for both vectors: $\varrho(X)=\varrho(Y)$. The following properties hold for the measure $\varrho_\sys(\cdot)$, defined in \eqref{def:syst-risk}. 
    \begin{proposition}
        \mbox{\phantom{m}}
    \begin{tightlist}{iii}
    \item If the univariate measures of risk $\varrho_i(\cdot)$, $i=1,\dots, m,$ are convex and $\varrho_0(\cdot)$ is convex and monotonic, then $\varrho_\sys(\cdot)$ is convex;
    \item If $\varrho_i(\cdot)$, $i=0,1,\dots, m,$ are nondecreasing, then $\varrho_\sys(\cdot)$ is nondecreasing as well;
    \item
    If $\varrho_i(\cdot)$, $i=0,1,\dots,m$, satisfy the translation property, then $\varrho_\sys(\cdot)$ satisfies this property as well. 
    If additionally $\varrho_i(0)=0$, $i=0,1,\dots,m$, then $\varrho_\sys(\cdot)$ satisfies the normalization properties,  $\varrho_\sys(\one) = 1$,
    $\varrho_\sys(0) = 0.$
    \item
    If $\varrho_i(\cdot)$, $i=0,1,\dots,m$, are positively homogeneous, so is $\varrho_\sys(\cdot)$, implying $\varrho_\sys(0) = 0.$
    \item If $\varrho_i (\cdot)$, $i=1,\dots,m$ are law-invariant risk measures, then $\varrho_\sys(X)$ is law-invariant.
    \end{tightlist} 
    \end{proposition}
    \begin{proof}
        (i) Given any $X,Y \in \Xf$ and $\alpha \in (0,1)$, we consider the random vector $W=\alpha X + (1-\alpha)Y$.  We have $\varrho_i(W_i) \leq \alpha\varrho_i(X_i) + (1-\alpha)\varrho_i(Y_i)$, $i = 1, \dots, m$. Defining a random variable $W^\prime$ on $\Omega_m$ by setting $W^\prime(i) = \alpha\varrho_i(X_i) + (1-\alpha)\varrho_i(Y_i)$, we obtain that $W\leq W^\prime$. Using the monotonicity and convexity of $\varrho_0$, we obtain
        \[
        \varrho_0(W)\leq \varrho_0(W^\prime) \leq \alpha\varrho_0[R(X)] + (1-\alpha)\varrho_0[R(Y)].
        \]
        Hence, $\varrho_\sys\big(\alpha X + (1-\alpha)Y\big) \leq \alpha \varrho_\sys(X) + (1-\alpha)\varrho_\sys(Y)$.
    
    (ii)  Suppose the vectors $X,Y \in \Xf$ satisfy $X_i \leq Y_i$, $i=1,\dots,m$, $P$-a.s.. Hence, $\varrho_i(X_i) \leq \varrho_i(Y_i)$ for all $i = 1, \dots, m$, by the monotonicity property of $\varrho_i$. This further implies that $R(X)\leq R(Y)$, entailing that  $\varrho_0[R(X)]\leq \varrho_0[R(Y)]$ due to the monotonicity of $\varrho_0(\cdot)$.

    (iii) Given a random vector $Z \in \Xf$ and a constant $a$, we have
        $[R({Z+a\one})](i) = \varrho_i(Z_i+a)= \varrho_i(Z_i)+a$. Hence, 
        $\varrho_0\big[R({Z+a\one})\big]= \varrho_0[R(Z)] + a.$

    (iv) It is sufficient to verify the positive homogeneity. Given a random vector $Z \in \Xf$ and $t > 0$, we have
    \[
        \varrho_\sys(tZ) = \varrho_0[R(tZ)] = \varrho_0\big(t R(Z)\big)=t\varrho_0[R(Z)]= t \varrho_\sys (Z),
    \]
    where we have used the positive homogeneity property of $\varrho_i(\cdot)$ for all $i=0,1, \dots, m$. 

    (v) If $X$ and $Y$ have the same distribution, it follows that their components $X_i$ and $Y_i$, $i=1,\dots,m$ have the same marginal distributions. Since risk measures $\varrho_i$ are law-invariant, we infer that $R(X)$ and $R(Y)$ have the same realizations, which entails that $\varrho_\sys[X] = \varrho_0[R(X)]= \varrho_0[R(Y)]=\varrho_\sys[Y].$ 
    \end{proof}

    We present two examples that show that the proposed non-linear aggregation of local risk implicitly enforces fairness of risk allocation to the units of the system.  
    \begin{example}
    {\rm 
    Consider the case when $\varrho_0$ is a convex combination of the expected value and the Average Value-at-Risk at some level $\alpha$ and all components of $Z$ are evaluated by the same measure of risk $\varrho(\cdot)$. Then for any $\kappa \in [0,1]$ and $c\in S_+^m$, the systemic measure of risk takes on the form:
    \begin{align*}
        \varrho_\sys^1 [X] & = \varrho_0[R(X)] = (1-\kappa) \Eb [R(X)] + \kappa \avar_\alpha [R(X)]\\
         & = (1-\kappa) \sum_{i=1}^m c_i \varrho[X_i] + \kappa \inf_{\eta \in \Rb} \Big\{  \eta + \frac{1}{\alpha} \sum_{i=1}^m c_i(\varrho[X_i] - \eta)_+ \Big\}.
    \end{align*}
    Here, the infimum with respect to $\eta \in \Rb$ is taken over the individual risks of the components $\varrho[X_i]$, $i = 1, \dots, m$. Hence, this method of aggregation imposes an additional penalty for the components whose risk exceeds some threshold. \hfill $\Box$
    }
    \end{example}
    \begin{example}
    {\rm
    Let $\varrho_0$ be the mean-upper-semideviation risk measure of the first order, and all components of $X$ be evaluated by the same measure of risk $\varrho[\cdot]$. We define the measure of systemic risk as follows:
    \begin{align*}
        \varrho_\sys^2[X] & = \varrho_0[R(X)] = \sum_{i=1}^m c_i \varrho[X_i] + \kappa \sum_{i=1}^m c_i \Big(\varrho[X_i] - \sum_{j=1}^m c_j \varrho[X_j] \Big)_+.
    \end{align*}
    Evidently, this systemic risk measure is an aggregation of the individual risk of the components, in which the risk of every component is compared to the weighted average risk of all components, and the deviation of the individual risk from that average is penalized.
    Hence, this method of non-linear aggregation maintains fairness within the system and keeps the components functioning at a similar level of risk. \hfill $\Box$
    }
    \end{example}

    Now, we shall establish the form of the subdifferential for the systemic risk measures, defined in \eqref{def:syst-risk}. This will also provide us with the additional description of the dual set $\Ac_\varrho= \partial \varrho[0] $. 

    For $V\in \Lc_p(\Omega_m, \Fc_m, c)$ and $\nu\in \Lc_q(\Omega_m, \Fc_m, c)$ with $\frac{1}{p}+ \frac{1}{q} =1 $, we use the notations $\langle \nu, V\rangle_c$ for the dual pairing between the two spaces, noticing that   
    \[
        \langle \nu, V\rangle_c = \sum_{i=1}^m c_i\nu_iV_i. 
    \]
    \begin{proposition}
    If $\varrho_i: \Lc_p(\Omega, \Fc, P)\to\Rb$, $i=1\dots m,$, and $\varrho_0: \Lc_p(\Omega_m, \Fc_m, c) \to \Rb$ are convex measures of risk for scalar-valued random variables, then $\varrho_\sys[\cdot]$ defined in \eqref{def:syst-risk} is subdifferentiable for any $X\in\Xf$ and its subdifferential has the form
    \[
        \partial \varrho_\sys(X) = \Big\{ 
        \zeta \in \Xf^* : \zeta_i = \nu_i \xi_i : ~ \nu \in \partial \varrho_0[R(X)], ~ \xi_i \in \partial\varrho_i (X_i)\text{ for all } i\in\Omega_m 
        \Big\}.
    \]
    The dual representation of $\varrho_\sys[\cdot]$ is given by \eqref{general dual} with the dual set
    \begin{equation*}
        \Ac_\sys = \Big\{ 
        \zeta \in \Xf^* : \zeta_i = \nu_i \xi_i \text{ a.s.}, ~ \nu \in \Ac_0, ~ \xi_i \in \Ac_i \text{ for all } i\in\Omega_m 
        \Big\},
    \end{equation*} 
    where $\Ac_i$ represents the dual set of the measure $\varrho_i(\cdot)$, $i=0, 1,\dots, m.$ 
    \end{proposition}
    \begin{proof}
    We consider the operator $F:\Xf\to \Lc_p(\Omega_m, \Fc, c)$, defined by $[F(X)](i)= \varrho_i(X_i)$ for $i\in \Omega_m.$ This operator is convex in the sense that for all $i\in\Omega_m$, we have 
    \begin{multline*}
    [F(\alpha X + (1-\alpha)Y)](i) = \varrho_i(\alpha X_i + (1-\alpha)Y_i)\leq \alpha \varrho_i(X_i) + (1-\alpha)\varrho_i(Y_i)\\  =
    \alpha F(X)](i) + (1-\alpha)F(X)](i).
    \end{multline*} 
    The operator is also norm-to-norm continuous at any $X\in\Xf$ due to the continuity of the convex measures of risk, where ``norm-to-norm'' means that we use the topology induced by the norms of $\Xf$ and $\Lc_p(\Omega_m, \Fc, c)$. To apply Theorem 3.7 in \cite{DDARriskbook}, we recall that a subgradient of $F(\cdot)$ at $X$ is a continuous linear operator $S:\Xf\to \Rb$ such that for all $Y\in \Xf$ we have $SY \le F(X+Y) - F(X)$. This is equivalent to $S\in \partial F(x)$ if and only if for all $Y\in \Xf$ we have $SY \le F'(X;Y).$
    Therefore, for any $X\in\Xf$ and $Y\in\Xf$, using the subdifferentiability properties of the convex measures of risk (\cite[Theorem 2.32]{DDARriskbook}), we obtain for all $\xi_i\in \partial \varrho (X_i)$, the following relations hold 
    \begin{gather*}
    \varrho_i(X_i) = \langle \xi_i, X_i\rangle= \int_\Omega \xi_i(\omega) X_i (\omega) P(d\omega)\quad \text{and} \\
    \varrho_i(X_i+Y_i) - \varrho_i(X_i)\geq \langle \xi_i, Y_i\rangle = \int_\Omega \xi_i(\omega) Y_i (\omega) P(d\omega).
    \end{gather*} 
    Additionally, since $\varrho_0[\cdot]$ is a convex risk measure imposed on a random variable with finitely many realizations, we infer
    \begin{gather*}
        \varrho_0[R(X)] = \sup_{\nu\in\Ac_0} \sum_{i=1}^m c_i \nu(i) R(X)(i)\quad\text{and}\\
        \partial\varrho_0[R(X)] = \big\{\nu\in\Ac_0: \langle \nu, R(X)\rangle_c = \varrho_0[R(X)] \big\}
    \end{gather*} 
    Hence, using $(\xi_1,\dots,\xi_m)$ with $\xi_i\in \partial \varrho (X_i)$ for all $i=1,\dots, m$ and $\nu\in\partial \varrho_0[R(X)] $, we obtain 
    \begin{multline}
        \varrho_0[R(X)] = \sum_{i\in\Omega_m} c_i \nu(i) R(X)(i) = \sum_{i\in\Omega_m} c_i \nu(i) \int_\Omega \xi_i(\omega) X_i (\omega) P(d\omega) \\
        = \int_\Omega  \sum_{i\in\Omega_m} c_i \nu(i) \xi_i(\omega) X_i (\omega) P(d\omega)
        = \int_\Omega  \langle \zeta(\omega) X (\omega)\rangle_c P(d\omega)
        \label{e:subdiff-rho-sys}
    \end{multline}
    For the particular case $\partial \varrho_\sys(0)$, we get
    \[
    \Ac_\sys = \partial \varrho_\sys(0) = \Big\{ 
        \zeta \in \Xf^* : \zeta_i = \nu_i \xi_i \text{ a.s.}, ~ \nu \in \Ac_0, ~ \xi_i \in \Ac_i \text{ for all } i\in\Omega_m 
        \Big\}.
    \]
    \end{proof}
   If we assume that $\varrho_i: \Lc_p(\Omega, \Fc, P)\to\Rb$, $i=1\dots m,$, and $\varrho_0: \Lc_p(\Omega_m, \Fc_m, c) \to \Rb$ are coherent measures of risk, then we notice that $\nu_i \geq 0$ and $\xi_i \geq 0$ a.s., which implies that $\zeta_i \geq 0$ a.s. as well. The relation
    \[
        \int_\Omega \zeta_i(\omega) P(d\omega) = \int_\Omega \nu_i \xi_i(\omega) P(d\omega) = \nu_i \int_\Omega \xi_i(\omega)P(d\omega) = \nu_i
    \]
    entails that 
    \[ 
        \langle \one, \mu_\zeta \rangle = \int_\Omega \langle \onen, \mu_\zeta \rangle P(d\omega) = \int_\Omega \sum_{i\in\Omega_m} \nu_i P(d\omega) = 1,
    \]
    as stated in \eqref{e:dualset-gen}, taking into account the normalization property. 

    We point out that the dual representation of the systemic risk measures shows that they provide a robust evaluation of potential losses when the underlying probability distributions are approximated or subjected to perturbations. 

    When $\varrho_\sys$ is a law-invariant risk measure, we can establish its consistency with stochastic dominance orders. Recall that an integrable random variable $V$ is stochastically smaller than an integrable random variable $Y$ with respect to the \textit{increasing convex order}, denoted $V \preceq_{\icx} Z$, if 
    \[ 
        \Eb[(V - \eta)_+] \leq \Eb[(Z-\eta)_+] \quad \text{for all } \eta \in \Rb. 
    \]
    We use the linear stochastic order for random vectors that is based on scalarization vectors from the simplex $S^m_+$  and the increasing convex order; it is defined as follows.
    \begin{definition}
        A random vector $X \in \Xf$ is linearly stochastically smaller than a random vector $Y \in \Xf$, denoted $X \preceq^{\lin}_{\icx} Y$, if
        \[
            c^\top X \preceq_{\icx} c^\top Y \quad \text{for all } c \in S^m_+.
        \]
    \end{definition}
    \begin{proposition}
    Assume that the space $(\Omega,\Fc,P)$ is either standard atomless or finite with equal probabilities of all simple events.
        If the functionals $\varrho_i(\Omega,\Fc,P)\to\Rb$, $i=1,\dots, m$ are law invariant coherent risk measures and $\varrho_0(\Omega_m,\Fc_c,c)\to\Rb$ satisfies the monotonicity property (A2), then the risk measure $\varrho_\sys(\cdot)$ defined in \eqref{def:syst-risk} is consistent with the linear increasing convex order. i.e., 
        \[  
            X \preceq^{\lin}_{\icx} Y \enskip  \enskip \Rightarrow \enskip \varrho_\sys[X] \leq \varrho_\sys[Y]. 
        \]
    \end{proposition}
    \begin{proof}
        Suppose $X \preceq^{\lin}_{\icx} Y$ holds. Observe that the linear order $X \preceq^{\lin}_{\icx} Y$ implies the coordinate-wise comparison $X_i \preceq_{\icx} Y_i$ for all $i=1,\dots.m$.
        Since any law invariant coherent risk measure $\varrho_i$ is consistent with the increasing convex order (\cite[Theorem 5.22]{DDARriskbook}), it follows that $\varrho_i[X_i] \leq \varrho_i[Y_i]$ for all $i=1,\dots,m$. This implies that $\varrho_0[R(X)] \leq \varrho_0[Y_R]$ by the monotonicity of $\varrho_0(\cdot)$. 
    \end{proof}

\section{Two-stage stochastic programming problem}
\label{s:problem}

    In this section, we formulate a two-stage stochastic programming problem for a system of agents where the risk of the system is evaluated using the risk measure proposed in Section \ref{s:analysis}. 
    Suppose the system consists of a set $\Jc=\{1,\dots,m\}$ of agents that can operate independently and communicate with their neighbors. The first-stage decisions, denoted as $x_i\in\Rb^{n_1}, ~ i\in\Jc$, are deterministic and made \textit{here and now}. Let $\Xc_i \subset \Rb^{n_1}$ be closed convex sets representing the constraints imposed on $x_i, \enskip i\in\Jc$ and $f_i:\Rb^{n_1}\to\Rb, \enskip i\in\Jc$ be convex continuous functions representing the cost associated with each $x_i$. 
    Additionally, the first-stage decisions of the agents are coupled through linear constraints using matrices $A_i \in \Rb^{d_1 \times n_1}, \enskip i\in\Jc$ as follows:
    \begin{equation*}
        \sum_{i\in\Ic} A_i x_i =b.
    \end{equation*}
    Without loss of generality, we may assume that $b=0$. If this is not the case, we may expand $x_i$ by an additional component, which is then sent to 1 within the sets $\Xc_i$; the matrices $A_i$ are extended by an additional column equal to $-c_ib.$
    
    The second-stage problem involves uncertainty modeled by random data $\xi$ with $N$ realizations, denoted as $\xi^s$, where $s\in\Sc=\{1,\dots,N\}$. Once the random data $\xi^s$ is observed, each agent $i\in\Jc$ makes a decision $y^s_i \in \Rb^{n_2}$ and the cost associated with it is given by a continuous function $g^s_i: \Rb^{n_2} \to \Rb$, which we assume to be linear. We model it as $g^s_i(y^s_i) = \langle q^s_i, y^s_i \rangle$ for some cost vector $q^s_i \in \Rb^{n_2}$. We assume that decision variables $y^s_i, \enskip i\in\Jc$ of the agents depend on their first-stage decisions $x_i$ and on common systemic variables $z^s$. This dependence is given by linear constraints:
    \begin{equation} 
        T^s_i x_i + \sum_{j\in\Jc} W^s_j y^s_j + B^s z^s = h_i^s, \quad i\in\Jc, s\in\Sc,
        \label{c:dynamics}
    \end{equation}
    where $T_i \in \Rb^{d_2\times n_1}$, $W^s_i \in \Rb^{d_2\times n_2}$, and $B^s \in {d_2\times d_3}$ for all $i\in\Jc, s\in\Sc$. The local constraints for $y^s_i$ are given by closed convex sets $\Yc^s_i \subset \Rb^{n_2}$ and the systemic variables are restricted by 
    $z^s\in Y^s\subseteq\Rb^{d_3}$. A cost function  $z^s\mapsto (u^s)^\top z^s$ is associated with the systemic variables $z^s\in Y^s\subseteq\Rb^{d_3}$ for each scenario $s\in\Sc$. 

    The vectors $h_i^s$ in the constraints can be included in \eqref{c:dynamics} by extending the decision variables $y^s_i$ by additional components equal to $1$ and choosing appropriate coefficients in $W^s_i$. Hence, we can assume again that $h_i^s=0$ without loss of generality. 

    Let $\varrho_\sys$ be a systemic risk measure defined as in \eqref{def:syst-risk}. The systemic risk reflects the total risk of individual agents, which also incorporates the risk associated with the purely systemic features gathered in the random vector $Z$ with realizations $z^s$. For all $i\in\Jc$, we introduce decision variables $r_i^s\in\Rb_+$ to determine the allocation of the systemic cost in scenario $s\in\Sc$ to each agent. We further define the random vector $R$ whose components $R_i$, $i\in\Jc$, have realizations $r_i^s$, $s\in\Sc$, a random variable $Q_i$ with realizations $ q^s_i$, $s\in\Sc$, and a random function $G(\cdot)$ that has $m$ realizations defined as $G_i(x,Y_i,Z,R_i) = \varrho_i[ \langle Q_i, Y_i \rangle + R_i ]$ for $i\in\Jc$. Since only finitely many realizations of the random elements exist, we can think of $G$ as a vector function with $m$ components. Due to the convexity and monotonicity of the measures $\varrho_\sys(\cdot)$ and $\varrho_i(\cdot)$, we can consider them as convex functions of the respective decision variables defined on a finite-dimensional space.  
    We formulate the following stochastic optimization problem:
    \begin{equation}
    \label{p:two-stage-pre}
        \begin{aligned}
            \min ~&~ \sum_{i\in\Jc} f_i(x_i) + \varrho_\sys[G], \\
            \text{s.t.} ~&~ \sum_{i\in\Jc} A_i x_i = 0, \\
            ~&~ T^s_i x_i + \sum_{j\in\Jc} W^s_j y^s_j + B^s z^s = 0, \quad i\in\Jc, s\in\Sc, \\
            ~&~ (u^s)^\top z^s = \sum_{i\in\Jc} r_i^s,\\
            ~&~ x_i \in \Xc_i, \;\; y^s_i \in \Yc^s_i, \;\; z^s\in Y^s,\;\; r_i^s \geq 0, \quad  i\in\Jc, s\in\Sc. 
        \end{aligned}
    \end{equation}
    Known approaches based on scenario decomposition and the multicut method are not applicable for this purpose because problem \eqref{p:two-stage-pre} contains coupling constraints linking all agents' decision variables, as well as constraints linking the decisions in all scenarios. Hence, the problem requires a separate dedicated analysis.  We consider that it is essential to solve problem \eqref{p:two-stage-pre} in a distributed way such that every agent can make its own decision and exchange a limited amount of information with the rest of the system to achieve optimality. 

    In the first step, we shall use the dual representation of the systemic risk measure 
    \begin{align*}
        \varrho_\sys[Q] & = \sup_{\substack{\nu\in\Ac_0\\ \xi_i\in\Ac_i,~i\in\Jc}} \sum_{s\in\Sc} \sum_{i\in\Jc} c_i\nu_i p_s \xi^s_i \big(\langle q^s_i, y^s_i \rangle + r_i^s\big) \\
        & = \sup_{\nu\in\Ac_0} \sum_{i\in\Jc} c_i\nu_i \sup_{\xi_i\in\Ac_i} \sum_{s\in\Sc} p_s \xi^s_i \big(\langle q^s_i, y^s_i \rangle + r_i^s\big). 
    \end{align*}
    Introducing auxiliary variables $\eta\in\Rb$, $\theta_i\in\Rb, ~i\in\Jc$, we change the objective of \eqref{p:two-stage-pre} to $ \sum_{i\in\Jc} f_i(x_i) +\eta$ and add the following constraints:
    \begin{gather}
        \eta \geq \sum_{i\in\Jc} c_i\nu_i \theta_i, \quad \nu \in \Ac_0, \label{c:sys-dual-eta} \\
        \theta_i \geq \sum_{s\in\Sc} p^s \xi^s_i (q^s_i)^\top y^s_i, \quad \xi_i \in\Ac_i, ~ i\in\Jc.
    \label{c:ind-dual-constr}
    \end{gather} 
    Further, we can approximate the set $\Ac_0$ by finitely many measures $\nu\in \Ac_0$ in the course of the numerical method using cutting planes of the following  form
    \begin{gather*}
        \eta \geq \sum_{i\in\Jc} c_i\nu_i^\ell \varrho_i \big(\langle q_i, Y_i \rangle + R_i\big),\quad \ell =1,2,\dots
    \end{gather*} 
    For brevity, let us assume that we deal with a finite approximation $\Bc$ of the set ${\Ac}_0$. Introducing auxiliary variables $w^\ell, ~ \ell\in \Bc$, we rewrite \eqref{c:sys-dual-eta} as follows:
    \begin{align}
        \sum_{i\in\Jc} c_i \nu^\ell_i \theta_i + w^\ell - \eta = 0, \quad \ell\in \Bc.
    \label{c:sys-dual-constr-eq}
    \end{align}
    To distribute \eqref{c:sys-dual-constr-eq}, we create copies of $\eta$, $w^\ell, ~\ell\in \Bc$ for every agent $i\in\Jc$, denoted $\eta_i$ and $w^\ell_i$, and set them equal between agents. We stack $\eta_i$ and $w^\ell_i, ~ \ell \in \Bc$ into a vector $v_i\in\Rb^{1+|\Bc|}$ for all $i\in\Jc$. Assuming that the system is a connected network, we create a spanning tree $\Ab$  of the network and set $v_i$ to be equal along the arcs $\Ab$ as follows:
    \begin{align*}
        v_i = v_j, \quad (i,j) \in \Ab.
    \end{align*}
    Additionally, we create a spanning tree $\Ab^s$ for each scenario $s\in\Sc$ and copies of the system's variables $z^s$ for every agent, denoted $z^s_i$, $i\in\Jc$. 
    \[
       z^s_i = z^s_j, \quad (i,j) \in \Ab^s,\; s \in \Sc.
    \]
    Then problem \eqref{p:two-stage-pre} can be rewritten as follows:
    \begin{align}
        \min ~&~ \sum_{i\in\Jc} f_i(x_i) + \sum_{i\in\Jc} c_i \eta_i, \label{p:obj} \\
        \text{s.t.} ~&~ \sum_{i\in\Jc} c_i \Big( \nu^\ell_i \theta_i + w^\ell_i - \eta_i \Big) = 0, \quad \ell\in \Bc, \label{p:constr1} \\
        ~&~ \theta_i \geq \sum_{s\in\Sc} p^s \xi^s_i (q^s_i)^\top y^s_i, \quad \xi_i \in\Ac_i, ~ i\in\Jc, \label{p:constr2} \\
        ~&~ \sum_{i\in\Jc} A_i x_i = 0, \label{p:constr3} \\
        ~&~ T^s_i x_i + \sum_{j\in\Jc}(W^s_j y^s_j + B^s z^s_j) = 0, \quad i\in\Jc, s\in\Sc, \label{p:constr4} \\
        ~&~ (u^s)^\top \sum_{j\in\Jc} c_i z^s_i = \sum_{j\in\Jc}   r_j^s, \quad  s\in\Sc,\label{p:two-stage-syscost}\\
        ~&~ v_i = v_j, \quad (i,j) \in \Ab, \label{p:constr5} \\
        ~&~ z_i^s = z_j^s, \quad (i,j) \in \Ab^s, \quad s\in\Sc\label{p:constr5zizj}\\
        ~&~ x_i \in \Xc_i, \;\; y^s_i \in \Yc^s_i, \;\; z_i^s\in Y^s,\;\; r_i^s \geq 0,\;\; w_i^\ell\geq 0, \quad i\in\Jc, s\in\Sc, \ell\in\Bc. \label{p:constr6}
    \end{align}
    We assign Lagrange multipliers $\lambda\in\Rb^{|\Bc|}$, $\alpha\in\Rb^{d_1}$, $\beta\in\Rb^{d_2}$, $\gamma\in\Rb$, $\delta\in\Rb^{|\Ab|}$, and $\sigma^s\in\Rb^{|\Ab^s|}$, $s\in\Sc$ to the coupling constraints \eqref{p:constr1}, \eqref{p:constr3}, \eqref{p:constr4}, \eqref{p:two-stage-syscost}, \eqref{p:constr5zizj}, \eqref{p:constr5}, and \eqref{p:constr5zizj}, respectively. 
    To further declutter notation, let us gather the primal variables in the vector $V_i=(x_i,v_i,Y_i,Z_i,R_i,\eta_i)$, $i\in\Jc$, and the dual variables $\zeta = (\alpha,\beta,\gamma,\delta,\sigma)$.

    The Lagrange function of problem \eqref{p:obj}--\eqref{p:constr6} takes on the form:
    \begin{equation}
    \label{p:LRfinal}
        \begin{aligned}
          \Lambda(V,\zeta) & = \sum_{i\in\Jc} (f_i(x_i)  + c_i \eta_i)
            + \sum_{\ell\in \Bc}  \lambda_\ell  \sum_{i\in\Jc}c_i(\nu_i^\ell \theta_i - \eta_i) + \langle \alpha, \sum_{i\in\Jc} A_i x_i \rangle \\
            & + \sum_{(i,j)\in\Ab} \delta_{ij} (v_i - v_j) + \sum_{s\in\Sc} p_s\bigg[ \sum_{i\in\Jc} \beta^s_i \Big( T^s_i x_i + \sum_{j\in\Jc} W^s_j y^s_j + B^s z_i^s\Big)  \\
            & + \gamma^s\big( \langle u^s, \sum_{i\in\Jc} c_i z_i^s\rangle - \sum_{i\in\Jc} r_i^s \big) + \sum_{(i,j)\in\Ab^s} \sigma_{ij}^s(z_i^s-z_j^s) \bigg].
        \end{aligned}
    \end{equation}
    The dual function, denoted $\Dc(\zeta)$ is given by 
    \begin{align*}
    \Dc(\zeta)  = 
    \min_{V,\theta} ~&~ \Lambda(V,\zeta)\\
             \text{s.t.} ~&~ \theta_i \geq \sum_{s\in\Sc} p^s \xi^s_i (q^s_i)^\top y^s_i, \quad \xi_i \in\Ac_i,\;  i\in\Jc, \\
            ~&~ x_i \in \Xc_i, \;\; y^s_i \in \Yc^s_i\!, \;\; z_i^s\in Y^s\!,\;\; r_i^s \geq 0,\;\; w_i^\ell\geq 0, \quad i\in\Jc, s\in\Sc, \ell\in\Bc. 
    \end{align*}
    The dual problem reads:
    \begin{equation}
      \label{p:dual_problem}  
      \max_{\zeta} \Dc(\zeta).
    \end{equation}
    For a given penalty parameter $\kappa > 0$, the augmented Lagrangian function is given by:
    \begin{equation}
        \begin{aligned}
            \Lambda_{\kappa_0} & (V,\zeta) 
            = \Lambda(V,\zeta)   + \frac{\kappa}{2} \sum_{\ell\in \Bc} \Big( \sum_{i\in\Jc} c_i \Big( \nu^\ell_i \theta_i + w^\ell_i - \eta_i \Big) \Big)^2 \\ 
            & + \frac{\kappa}{2} \bigg\| \sum_{i\in\Jc} A_i x_i \bigg\|^2 +  \frac{\kappa}{2} \sum_{(i,j)\in\Ab} \|v_i - v_j\|^2 
              + \frac{\kappa}{2} \sum_{s\in\Sc} \sum_{(i,j)\in\Ab^s} p_s\|z_i^s - z_j^s\|^2\\
            & + \frac{\kappa}{2} \sum_{s\in\Sc}  \sum_{i\in\Jc} p_s \bigg\| T^s_i x_i + \sum_{j\in\Jc} W^s_j y^s_j + B^s z_i^s\bigg\|^2 + \frac{\kappa}{2} \sum_{s\in\Sc}  p_s \Big\| \langle u^s,z_i^s \rangle -\sum_{j\in\Jc} r_i^s\Big\|^2. 
        \end{aligned}
        \label{p:global-lagrangian}
    \end{equation}
    We denote $\Delta_i = \sum_{(i,j),(j^\prime,i)\in\Ab} (\delta_{ij} - \delta_{j^\prime i})$, $\pi_i^s = \sum_{(i,j),(j^\prime,i)\in\Ab^s} (\sigma_{ij}^s - \sigma_{j^\prime i}^s)$, and $\bar{\beta}_i = \sum_{j\in\Jc} \beta^s_j$. The global Augmented Lagrangian \eqref{p:global-lagrangian} is replaced by  \textit{local Augmented Lagrangians} $\Lambda^i_{\kappa_0}$ and every agent $i$ solves an optimization problem formulated as follows
    \begin{equation}
    \begin{aligned}
        \min_{V_i} ~&~ \Lambda^i_{\kappa_0}(V_i, \bar{V}, \zeta) \\
        \text{s.t.} ~&~ \theta_i \geq \sum_{s\in\Sc} p^s \xi^s_i (q^s_i)^\top y^s_i, \quad \xi_i \in\Ac_i, \\
        ~&~ x_i \in \Xc_i, \;\; y^s_i \in \Yc^s_i, \;\; z_i^s\in Y^s,\;\; r_i^s \geq 0,\;\; w_i^\ell\geq 0, \quad i\in\Jc, s\in\Sc, \ell\in\Bc, 
    \end{aligned}
    \label{p:new-ind-problem}
    \end{equation}
where the local Lagrangians have the form:
    \begin{align*}
        \Lambda_{\kappa}^i & (V_i,\bar{V},\zeta) 
        = f_i(x_i) + c_i \eta_i
        + \sum_{\ell\in \Bc} \lambda_\ell c_i \Big( \nu^\ell_i \theta_i + w^j_i - \eta_i \Big) + \langle \alpha, A_i x_i \rangle + \Delta_i v_i\\ 
        & + \sum_{s\in\Sc} p_s\beta^s_i (T^s_i x_i + B^s z_i^s) 
         + \sum_{s\in\Sc} p_s\bar{\beta}^s W^s_i y^s_i  + \sum_{s\in\Sc} p_s \gamma^s \big(c_i\langle u^s, z_i^s\rangle - r_i^s \big) \\ 
        & + \sum_{s\in\Sc} p_s \pi_i^sz_i^s + \frac{\kappa}{2} \sum_{(i,j)\in\Ab} \|v_i - \bar{v}_j\|^2  
         + \frac{\kappa}{2} \sum_{\ell\in \Bc } \Big( c_i \nu^\ell_i \theta_i + w^\ell_i - \eta_i  + \sum_{\substack{k\in\Jc\\k\neq i}} c_k\nu^\ell_k \bar{\theta}_k  \Big)^2 \\
        &
        + \frac{\kappa}{2} \sum_{i\in\Jc} \sum_{s\in\Sc} p_s\Big\| T^s_i x_i + W^s_i y^s_i + B^sz_i^s + \sum_{\substack{j\in\Jc\\j\neq i}} W^s_j \bar{y}^s_j \Big\|^2 + \frac{\kappa}{2} \Big\| A_i x_i + \sum_{\substack{j\in\Jc\\j\neq i }} A_j \bar{x}_j \Big\|^2 \\ 
        & + \frac{\kappa}{2} \sum_{s\in\Sc}  p_s \Big\|\langle u^s,z_i^s \rangle -r_i^s - \sum_{\substack{j\in\Jc\\j\neq i}} \bar{r_i^s}\Big\|^2 
        + \frac{\kappa}{2} \sum_{s\in\Sc} \sum_{(i,j)\in\Ab^s} p_s\|z_i^s - z_j^s\|^2 .
    \end{align*}

    We propose the following decomposition method with parameters $\kappa >0$ associated with the augmented Lagrangian and $\tau>0$ for updates of the primal variables. 

\begin{minipage}{0.96\textwidth}
   \begin{tightitemize}
    \item[]   \vspace{0.2cm} \hrule \vspace{0.2cm}
                \textbf{Systemic Risk Optimization with Nonlinear Scalarization} 
                \vspace{0.2cm} \hrule \vspace{0.2cm}  
    \item[Step 0.] Set $k=1$ and define initial primal variables $V^1$ and dual variables $\zeta^1$.
    \item[Step 1.] For every $i\in\Jc$, solve problem \eqref{p:new-ind-problem} and denote its optimal solution $\bar{V}^k_i$.
    \item[Step 2.] For every $i\in\Jc$, update the primal variables
            \[
            V^{k+1}_i = V^k_i + \tau (\bar{V}^k_i - V^k_i).
            \]
    \item[Step 3.]
    If the coupling constraints \eqref{p:constr1}, \eqref{p:constr3}, \eqref{p:constr4}, \eqref{p:two-stage-syscost}, \eqref{p:constr5zizj}, \eqref{p:constr5}, and \eqref{p:constr5zizj} are satisfied at $V^{k+1}$, then stop. Otherwise, update the dual variables
        \begin{align*}
        & \lambda^{k+1}_\ell= \lambda^k_\ell + \kappa\tau \sum_{i\in\Jc} c_i \Big( \nu^\ell_i \theta^k_i + w^\ell_i - \eta^k_i \Big), \quad \ell \in \Bc, \\
        & \alpha^{k+1}= \alpha^k_i + \kappa\tau \sum_{i=1}^m A_i x^{k+1}_i, \\
        & \beta^{s,k+1}_i= \beta^{s,k}_i + \kappa\tau \bigg( T^s_i x^{k+1}_i + \sum_{j\in\Jc} W^s_j y^{s,k+1}_j \bigg), \quad i\in\Jc, s\in\Sc, \\
        & \gamma^{s,k+1}= \gamma^{s,k} + \kappa\tau \sum_{i\in\Ic} \big( c_i(u^s)^\top z_i^{s,k+1} -  r_i^{s,k+1} \big), \quad s\in\Sc, \\
        & \delta^{k+1}_{ij}= \delta^{t}_{ij} + \kappa\tau (v^{k+1}_i - v^{k+1}_j), \quad (i,j)\in\Ab, \\
        & \sigma^{s,k+1}_{ij}= \sigma^{s,k}_{ij} + \kappa\tau (z^{s,k+1}_i - z^{s,k+1}_j), \quad (i,j)\in\Ab^s,    
        \end{align*}
        and go to Step 1.
     \vspace{0.2cm} \hrule \vspace{0.2cm}
    \end{tightitemize}
    \end{minipage}

    We note that the solution of problem \eqref{p:new-ind-problem} in Step 1 can be accomplished by another decomposition method for a risk-averse two-stage problem as discussed in \cite{Risk-Multicut-Miller}. 

    \begin{proposition}
        Let $\Xc_i$, $\Yc^s_i$, and $Y^s$ be convex compact sets for all $i\in\Jc$ and $s\in\Sc$. If the parameter $\tau$ satisfies $0 < \tau < \frac{1}{m}$, the decomposition method for nonlinear risk scalarization converges to the optimal solution of the problem \eqref{p:obj}-\eqref{p:constr6}.    
    \end{proposition}
    \begin{proof}
        By the convergence properties of ADAL defined in Theorem 2 in \cite{ADAL}, if the stepsize $\tau$ satisfies $0 < \tau < \frac{1}{m}$, the algorithm converges to the optimal solution of the dual problem to \eqref{p:obj}-\eqref{p:constr6}; it generates a sequence of dual variables $\{\lambda^{k},\alpha^k,\beta^{s,k},\gamma^{s,k}, \delta^k, \sigma^{s,k} \}$, $s\in\Sc$ converging to the optimal solution of it.
        The sequences $\{ x_i^k, y_i^{s,k}, z_i^{s,k}\}$ are bounded for all $i\in\Jc$  and hence, they each have an accumulations point  
        $(\hat{x}_i, \hat{y}_i{s}, \hat{z}_i^{s})$. Since the constraints are satisfied in the limit, we have 
        $\eta_i^k  \approx \eta_j^k $ for all $i,j\in\Jc$ and 
        $\sum_{i\in\Jc} c_i \nu^\ell_i \theta_i^k + w^{\ell,k}_i \approx \eta_i^k$.  Due to the constraint $w^{\ell,k}_i\geq 0$ and $\eta_i^k$ being minimized, we conclude that $\theta_i^k$ is also minimized indirectly. Hence, $\theta_i^k = \sum_{s\in\Sc} p^s \xi^s_i (q^s_i)^\top y^s_i$ for some $\xi_i \in\Ac_i$. This implies that the sequences $\{\theta_i^k\}$ also have accumulation points $\hat{\theta}_i$  for all $i\in\Jc$. By the same argument $\sum_{i\in\Jc} c_i \nu^\ell_i \theta_i^k = \eta_i^k$ for some $\ell \in \Bc$. This in turn entails that the sequences $\{\eta_i^k\}$have accumulation points $\hat{\eta}_i$ for all $i\in\Jc$. Returning to constraint \eqref{p:constr1}, we infer the existance of accumulation points for the sequences $\{ w^{\ell,k}_i\}$,  $i\in\Jc$. Finally, we observe that $0\leq r_i^{s,k}\leq (u^s)^\top z_i^{s,k}$. The boundedness of the variables $z_i^{s,k}$ enforces that those sequences are bounded as well, which entails that $\{ r_i^{s,k}\}$ have accumulation points for all $i\in \Jc.$ 
        Hence, the sequence of the entire set of primal variables $\{ x_i^k, v_i^k, y_i^{s,k}, z_i^{s,k}, r_i^{s,k}, \eta_i^k,\theta_i^k \}$, $i\in\Jc$ generated by this algorithm has an accumulation point, which we can ensure by passing to subsequences.  Any such point is an optimal solution of \eqref{p:obj}-\eqref{p:constr6} due to the properties of the augmented Lagrangian method.
    \end{proof}

\section{Numerical Experiments}
\label{s:numerical_results}

    We apply the proposed numerical method to a disaster management problem similar to the two-stage problem considered in \cite{Noyan-disaster-problem}. The disaster management problem deals with the optimal allocation of resources at select facilities to minimize the damage caused by a random disaster. In the first stage of the problem, we decide the amount of resources to be pre-allocated at some facility locations. Once a disaster occurs at a random location, the resources can be moved between the facilities to satisfy the realized demand. 

    Let $\Jc$ denote the set of facility locations (nodes) indexed by $i$ and $\Nc_i$ denote the set of neighboring nodes of $i\in\Jc$. First, we determine the amount of resources $r_i$ to be positioned at each location $i \in \Jc$: the \emph{here-and-now} decisions. The cost associated with allocating one resource unit at location $i \in \Jc$ is denoted by $d_i$. A limited budget of $M>0$ resource-units is given.

    Once we observe the location of the disaster and a random demand $D_i$ for resources at locations $i\in\Jc$, we decide on the redistribution of resources. We approximate the uncertainty by a set $\Sc$ of $N$ scenarios, i.e., $\Sc = \{1,\dots,N\}$. Once the location of the disaster and the corresponding demand levels at the facilities are observed, we make the following decisions: $x^{s}_{ij}$ is the amount of resources shipped from $i\in\Jc$ to $j\in\Nc_i$, $u^{s}_i$ is the amount of resources not used at location $i\in\Jc$, and $z^{s}_i$ is the amount of shortage of resources at locations $i\in\Jc$ in scenario $s\in\Sc$. The cost of shipping one unit from $i\in\Jc$ to $j\in\Nc_i$ is given by $a_{ij}^{s}$, unit salvage cost for a resource is $h_i$, and unit penalty cost for the shortage is $b_i$ for $i\in\Jc$. Then, every node $i\in\Jc$ minimizes its cost given as
    \[
        Q^s_i = \sum_{j\in\Jc} a^{s}_{ij} x^{s}_{ij} + h_i u^{s}_i + b_i z^{s}_i.
    \]
    The proportion of pre-allocated resources that can be used is given by $\gamma^s_i \in [0,1]$ for $i\in\Jc$ and the number of resources that can be sent from $i\in\Jc$ to $j\in\Nc_i$ is limited by the maximum capacity of the link $U^s_{ij}$. 
    
    We use the mean semi-deviation risk measure for the individual risk measurement as well as in the systemic risk aggregation. 
    The disaster management problem is then formulated as follows:
    \begin{align}
        \min ~&~ \sum_{i\in\Jc} d_i r_i + \sum_{i\in\Jc} c_i (\theta_i + \varkappa_0 \vartheta_i) \label{ch6:num:p-objective} \\
        \text{s.t.} ~&~ \vartheta_i \geq \theta_i - \sum_{j\in\Jc} c_j \theta_j, \quad i\in\Jc, \label{ch6:num:p-risk-shortfalls} \\
        ~&~ \theta_i = \sum_{s\in \Sc} p^s Q^s_i + \varkappa_i \sum_{s\in \Sc} p^s v^s_i, \quad i\in\Jc, 
        \label{ch6:num:p-risk-vals} \\
        ~&~ v^s_i \geq Q^s_i - \sum_{\ell\in\Sc} p^\ell Q^\ell_i, \quad i\in\Jc, \; s\in\Sc, \label{ch6:num:p-shortfalls} \\
        ~&~ \gamma^s_i r_i + \sum_{j\in\Jc} x^s_{ji} - \sum_{j\in\Jc} x^s_{ij} - D^s_i = u^s_i - z^s_i, \quad i\in\Jc, \; s\in\Sc, \label{ch6:num:p-flow-constr} \\
        ~&~ \sum_{i\in\Jc} r_i \leq M, \label{ch6:num:p-max-resources} \\
        ~&~ 0\leq x^s_{ij} \leq U^s_{ij}, \quad i\in\Jc, \; j \in \Nc_i, \; s\in\Sc, \\
        ~&~ z^s_i \geq 0, \; u^s_i \geq 0, \; r_i \geq 0, \; \vartheta_i\geq 0\;  v_i^s\geq 0,    \quad i\in\Jc, \; s\in\Sc. 
        \label{ch6:dis-problem}
    \end{align}
    In this case, \eqref{ch6:num:p-risk-shortfalls}, \eqref{ch6:num:p-flow-constr}, and \eqref{ch6:num:p-max-resources} are coupling constraints and need to be relaxed using Lagrange multipliers. 

   The Lagrangian relaxation takes on the form:
    \FloatBarrier\begin{align*}
        \min ~&~ \sum_{i\in\Jc} \bigg[ d_i r_i +  c_i (\theta_i + \varkappa_0 \vartheta_i) + \alpha\big(r_i -c_iM) +
                 \beta_i(\theta_i - \vartheta_i ) - c_i\theta_i \sum_{j\in\Jc}\beta_j  \\
             ~&~ \sum_{s\in \Sc} \Big(\delta_i^s(\gamma^s_i r_i - D^s_i +z^s_i - u^s_i) + \Big(\delta_i^s - \sum_{j\in\Nc_i}\delta_j^s\Big)x^s_{ji} + \Big(\sum_{j\in \Nc_i}\delta_j^s - \delta_i^s\Big)x^s_{ij}\Big) \bigg]\\
         \text{s.t.} 
        ~&~ \theta_i = \sum_{s\in Sc} p^s Q^s_i + \varkappa_i \sum_{s\in Sc} p^s v^s_i, \quad i\in\Jc, \\
        ~&~ v^s_i \geq Q^s_i - \sum_{t=1}^N p^t Q^t_i, \quad i\in\Jc, \; s\in\Sc, \\
        ~&~ 0\leq x^s_{ij} \leq U^s_{ij}, \quad i\in\Jc, \; j \in \Nc_i, \; s\in\Sc, \\
        ~&~ z^s_i \geq 0, \; u^s_i \geq 0, \; r_i \geq 0, \; \vartheta_i\geq 0, \;  v_i^s\geq 0, \quad i\in\Jc, \; s\in\Sc.    
    \end{align*}
    Evidently, it splits into subproblems for each location $i\in\Jc$.

    We solve the problem for a network of $5$ facilities at fixed locations in a $1 \times 1$ square map. To generate scenarios, we pick a random location in the map where a disaster occurs and calculate the demand levels at the facilities depending on the disaster location. We assume that the facilities closer to the center of the disaster have a higher demand for resources and vice versa. We calculate the demand at location $i\in\Jc$ in scenario $s\in\Sc$ as follows:
    \[
        D^s_i = \frac{\nu_1}{1+e^{\nu_2 \|\Delta^s_{i}\|}},
    \]
    where $\nu_1$ and $\nu_2$ are some constants, and $\|\Delta^s_i\|$ is the distance between the facility and the location of the disaster in scenario $s$. In our experiments, we set $\nu_1 = 20$, $\nu_2 = 2$, and the arc capacity values $U_{ij} = 1.5$ for all $i\in\Jc$ and $j\in\Nc_i$. We set the cost values as $d_i = 0$, $h_i = 5$, $b_i = 5$ for all $i\in\Jc$ and  $a_{ij} = 1$ for $i\in\Jc, j\in\Nc_i$. The amount of total resources available is $M = 25$ and $\gamma^s_i = 0.95$ for all $i\in\Jc$ and $s\in\Sc$. The scalarization vector $c$ is fixed, and its components are set $c_i = \frac{1}{5}$ for all $i\in\Jc$.

    We solve the problem for $10, 50$, and $100$ scenarios using two risk aggregation methods. First, we aggregate the individual risk measures $\varrho_i$ using the weighted summation. This aggregation corresponds to setting $\varrho_0[\cdot] = \Eb[\cdot]$. To solve this problem, we set $\varkappa_0 = 0$, remove the first constraint in \eqref{ch6:dis-problem} with its corresponding Lagrange multipliers, and implement the proposed numerical method. Next, we solve the problem by setting $\varrho_0[\cdot]$ as the mean-upper-semideviation of order $1$ with $\varkappa_0 = 0.5$. All individual risk measures $\varrho_i[\cdot]$, $i\in\Jc$ are also set to be mean-upper-semideviation of order $1$ with $\varkappa_i = 0.5$. The method is implemented in Python using Gurobi solver on a PC with a 10-core CPU and 3.2 GHZ.

    The results are summarized in Tables~\ref{tab:individual_risk_values_and_the_total_risk_for_aggregation_of_risks} and \ref{tab:optimal_allocation_of_resources_for_two_aggregation_methods}. The notation $\varrho_c$ refers to the risk measure defined in 
    \eqref{rho_max_S} with $S=\{c\}$  and $c =c_i = \frac{1}{5}$ for $i=1,\dots,5$. Linear scalarization in Table~\ref{tab:optimal_allocation_of_resources_for_two_aggregation_methods} refers to the evaluation of total risk by the measure $\varrho_c$.  
\begin{center}
    \begin{table}
        \begin{tabular}{|c|c|c|c|c|c|c|c||c|}
            \hline
            N & $\varrho_0$ & $\varrho_1$ & $\varrho_2$ & $\varrho_3$ & $\varrho_4$ & $\varrho_5$ & $\varrho_\sys$ & $\varrho_c$ \\ \hline
             10 & Expected Value & {20.49} & 9.44 & {6.27} & 16.46 & 6.84 & {11.90} & {11.36} \\ 
            & Mean-Upper-Semidevation & {13.67} & 12.35 & 12.35 & 12.35 & {11.03} & {12.48} & \\ \hline
            50 & Expected Value & {16.41} & 15.37 & {8.01} & 16.4 & 12.28 & {13.69} & {13.16} \\ 
            & Mean-Upper-Semidevation & 13.77 & 13.77 & 13.77 & 13.77 & 13.77 & {13.77} & \\ \hline
            100 & Expected Value & {17.16} & 17.09 & {8.55} & 14.81 & 11.09 & {13.74} & {13.12} \\ 
            & Mean-Upper-Semidevation & 13.79 & 13.79 & 13.79 & 13.79 & 13.79 & {13.79} & \\ \hline
        \end{tabular}
        \caption{Individual risk values and the total risk for aggregation methods
         using the weighted sum (Expected Value), nonlinear aggregation of risks (Mean-Upper-Semideviation), and an additional column for the total risk evaluation by $\varrho_c(\cdot)$.}
         \label{tab:individual_risk_values_and_the_total_risk_for_aggregation_of_risks}
    \end{table}
       \begin{table}
        \centering
        \begin{tabular}{@{}|c|c|c|c|c|c|c|@{}} 
            \hline
            N &  Aggregation & $r_1$ & $r_2$ & $r_3$ & $r_4$ & $r_5$ \\ \hline \hline
            10 & $\varrho_0$ is Expected Value & 3.57 & 5.33 & 6.22 & 4.85 & 5.04 \\ 
            &$\varrho_0$ is Mean-Upper-Semidevation & 5.77 & 4.49 & 4.66 & 5.23 & 4.85 \\ 
            & $\varrho_c$  Linear Scalarization & 3.43 & 5.10 & 6.33 & 4.99 & 5.16 \\ \hline
            50 & $\varrho_0$ is Expected Value & 3.75 & 4.33 & 6.77 & 4.67 & 5.49 \\ 
            & $\varrho_0$ is Mean-Upper-Semidevation & 4.46 & 4.89 & 5.59 & 4.85 & 5.21 \\ 
            & $\varrho_c$  Linear Scalarization & 3.58 & 4.52 & 6.58 & 4.84 & 5.49 \\ \hline
            100 & $\varrho_0$ is Expected Value & 3.82 & 4.02 & 6.46 & 5.01 & 5.70 \\ 
            & $\varrho_0$is Mean-Upper-Semidevation & 4.70 & 4.69 & 5.55 & 5.01 & 5.05 \\ 
            & $\varrho_c$  Linear Scalarization & 3.81 & 4.19 & 6.21 & 5.30 & 5.48 \\ \hline
        \end{tabular}
        \caption{Optimal allocation of resources for two risk aggregation methods and linear scalarization.}
        \label{tab:optimal_allocation_of_resources_for_two_aggregation_methods}
    \end{table} 
\end{center}
 We observe that solving the problem using the weighted sum of the individual risks leads to a wide range of risk values at the facilities. For example, for $N=10$, the individual risk values range from $6.84$ to $20.49$, and the total risk of the system is $11.90$. Meanwhile, using the mean-upper-semideviation of order $1$ to aggregate the risk measures results in a smaller difference between individual risk values, ranging from $11.03$ to $13.67$, with a total risk value of $12.615$. Even though the risk of the system increased by a small amount, the proposed risk aggregation method introduced fairness among the members of the system. The same effect is observed for $50$ and $100$ scenarios.
    We also emphasize that the risk measures provide robustness to model uncertainty and approximations, which is evidenced by their dual representation. 
    This issue is important in stochastic optimization as we approximate the relevant distributions by finite number of scenarios and our models are frequently data-driven. 

    The numerical performance of the method for those experiments is reported in the appendix.

\section{Conclusions}

We provide a framework for risk evaluation and control for complex distributed stochastic systems, in which the evaluation of risk of the system's components is based on confidential and/or proprietary information. The proposed approach facilitates fair risk allocation to agents or units.
 
The systemic risk measures are based on a sound axiomatic foundation while also enjoying constructions that facilitate risk-averse decision-making by distributed numerical methods. We have provided new theoretical results regarding the form of the subdifferential of the systemic measures involving the non-linear risk aggregation.

We formulate a two-stage decision problem for a distributed system using a systemic measure of risk with non-linear risk aggregation. A new decomposition method for solving the problem is provided. 

We apply the method to a problem arising in disaster management, where we have paid attention to maintain fair risk allocation to all affected locations. Our numerical results show the viability of the proposed methodology. The most important observation of our numerical experiments is that the nonlinear aggregation of individual risks ensures fairness in risk allocation while at the same time providing robustness to the inaccuracies of the uncertainty model.

\section*{Acknowledgements}
The authors thank the Associate Editor and the anonymous referees whose comments and suggestions helped improve the paper. 
The support of the Air Force Office of Scientific Research under award FA9550-24-1-0284 is gratefully acknowledged.

\section*{Declarations}

\begin{itemize}
\item Conflict of interest/Competing interests: 
The authors declare that they have no conflict of interest.
\item Data availability: no specific data sets are used; the data for the numerical experiments is randomly generated as described in the text.
\end{itemize}

\begin{appendices}

\section{Data for the numerical performance of the decomposition method}

The performance of the decomposition method for two risk aggregations is given in Table \ref{table:results}. The convergence results for the weighted sum of the risk measures are shown in Figures \ref{fig:convergence-sum-risks-10} and \ref{fig:convergence-sum-risks-100} for $10$ and $100$ scenarios, respectively. Similarly, the convergence results for the aggregation of the risk measures using the mean-upper-semideviation of order $1$ are shown in Figures \ref{fig:convergence-nonlinear-risks-10} and \ref{fig:convergence-nonlinear-risks-100} for $10$ and $100$ scenarios. 

    \begin{table}
    \centering
        \begin{tabular}{|c|c|c|c|}
            \hline
            & & \textbf{Weighted sum} & \textbf{Mean-Upper-Semideviation} \\ \hline 
            N = 10 & $\kappa$ & 0.3 & 0.3 \\ 
            & $\kappa_0$ & 0.01 & 0.01 \\ 
            & Iterations & 1069 & 475 \\ 
            & Time, s & 474 & 213 \\ \hline 
            N = 50 & $\kappa$ & 0.3 & 0.3 \\ 
            & $\kappa_0$ & 0.01 & 0.01 \\ 
            & Iterations & 1282 & 2581 \\ 
            & Time, s & 2776 & 5567 \\ \hline 
            N = 100 & $\kappa$ & 0.3 & 0.3 \\ 
            & $\kappa_0$ & 0.005 & 0.005 \\ 
            & Iterations & 1962 & 2473 \\ 
            & Time, s & 8511 & 10617 \\ \hline
        \end{tabular}
        \caption{\small Numerical performance of the distributed method for two aggregation methods: $\varrho_0$ is the expected value (weighted sum) or $\varrho_0$ is the mean-upper-semideviation of order 1.}
        \label{table:results}
    \end{table}

It can be seen that solving the problem using the weighted sum of the individual risks leads to a wide range of risk values at the facilities. For example, for $N=10$, the individual risk values range from $6.84$ to $20.49$, and the total risk of the system is $11.90$. Meanwhile, using the mean-upper-semideviation to aggregate the risk measures results in a smaller difference between individual risk values, ranging from $11.03$ to $13.67$, with a total risk value of $12.615$. Even though the risk of the system increased by a small amount, the proposed risk aggregation method introduced fairness among the members of the system. The same effect can be observed for the case of $50$ and $100$ scenarios.
    \begin{figure}[H]
        \centering
        \subfloat[]{\includegraphics[width=0.45\textwidth]{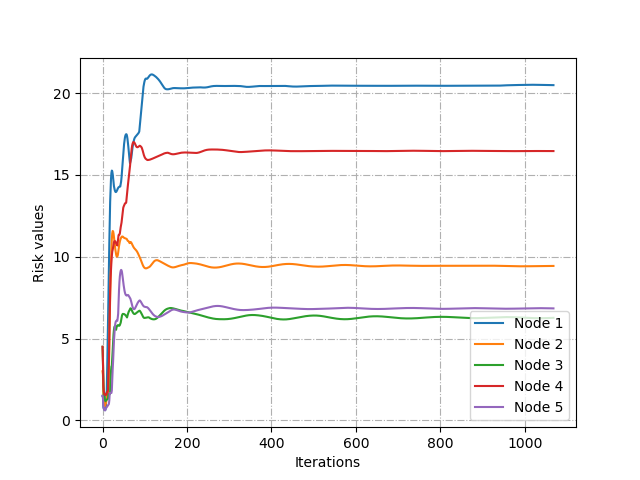}}
        \hfill
        \subfloat[]{\includegraphics[width=0.45\textwidth]{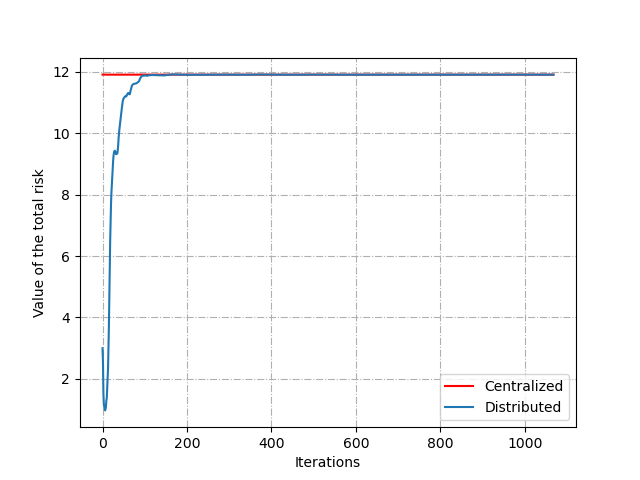}}
        \caption{(a) Evolution of the individual risk values $\varrho_i$ and (b) convergence of the total risk value to the centralized value. In this case, $\varrho_0$ is the expected value and $N = 10$.}
        \label{fig:convergence-sum-risks-10}
    \end{figure}
\vspace*{-10ex}
    \begin{figure}[H]
        \centering
        \subfloat[]{\includegraphics[width=0.45\textwidth]{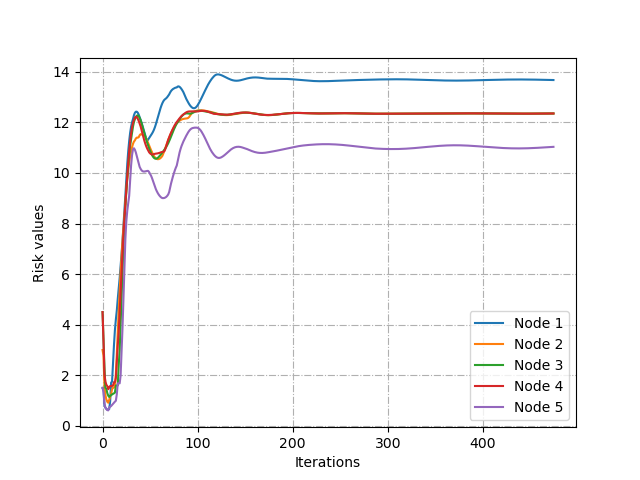}}
        \hfill
        \subfloat[]{\includegraphics[width=0.45\textwidth]{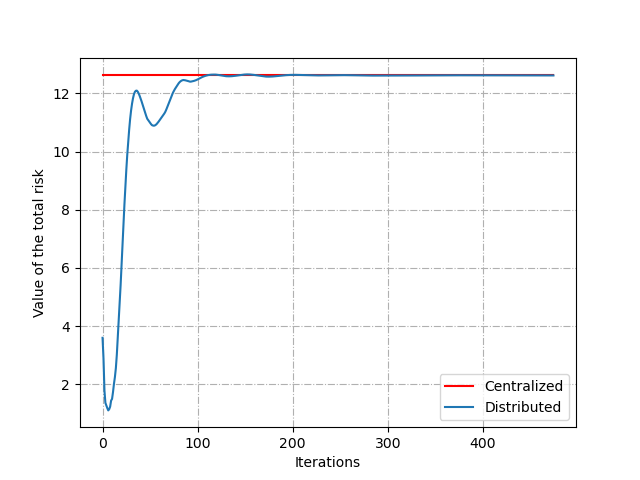}}
        \caption{(a) Evolution of the individual risk values $\varrho_i$ and (b) convergence of the total risk value to the centralized value. In this case, $\varrho_0$ is the mean-upper-semideviation of order 1 and $N = 10$.}
        \label{fig:convergence-nonlinear-risks-10}
    \end{figure}
\vspace*{-10ex}
    \begin{figure}[H]
        \centering
        \subfloat[]{\includegraphics[width=0.45\textwidth]{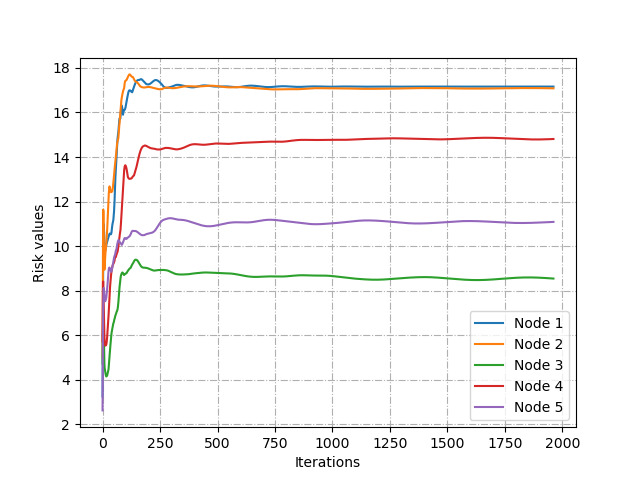}}
        \hfill
        \subfloat[]{\includegraphics[width=0.45\textwidth]{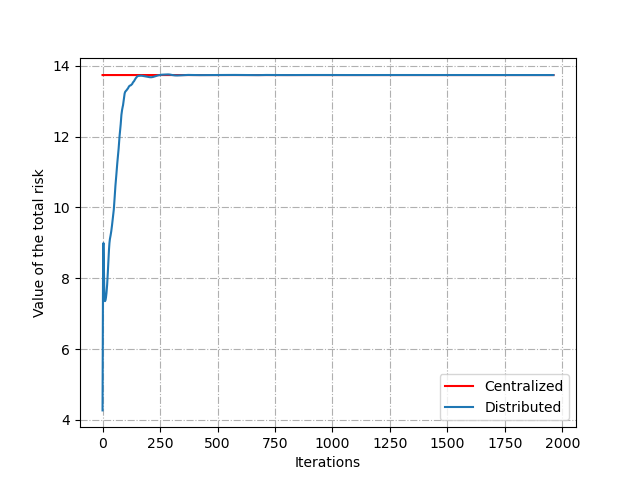}}
        \caption{(a) Evolution of the individual risk values $\varrho_i$ and (b) convergence of the total risk value to the centralized value. In this case, $\varrho_0$ is the expected value and $N = 100$.}
        \label{fig:convergence-sum-risks-100}
    \end{figure}
\vspace*{-10ex}
    \begin{figure}[H]
        \centering
        \subfloat[]{\includegraphics[width=0.45\textwidth]{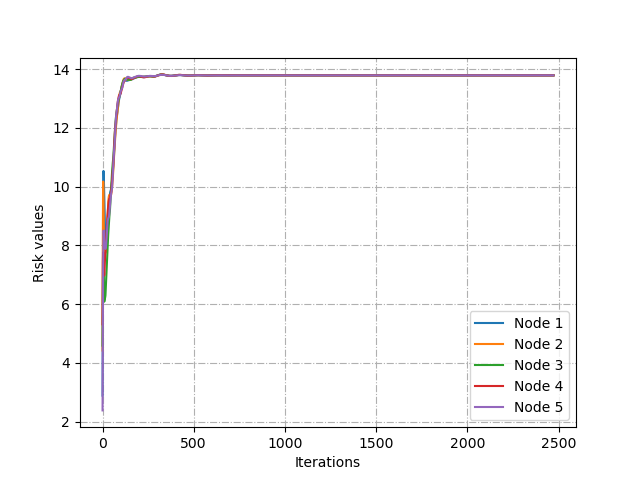}}
        \hfill
        \subfloat[]{\includegraphics[width=0.45\textwidth]{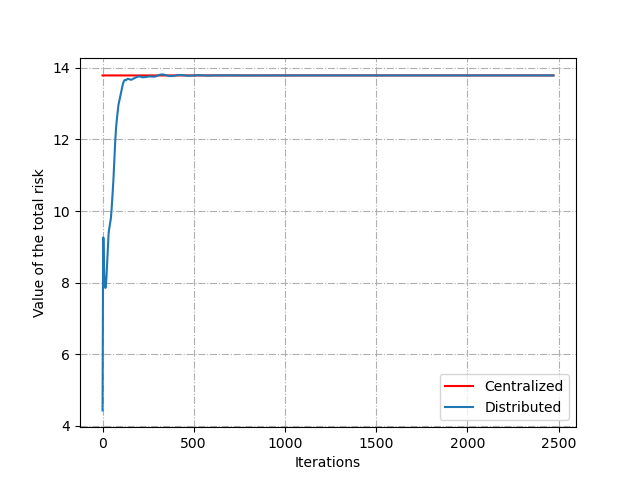}}
        \caption{(a) Evolution of the individual risk values $\varrho_i$ and (b) convergence of the total risk value to the centralized value. In this case, $\varrho_0$ is the mean-upper-semideviation of order 1 and $N = 100$.}
        \label{fig:convergence-nonlinear-risks-100}
    \end{figure}

\end{appendices}

\bibliographystyle{plainnat}

\end{document}